\documentclass{amsart}
\usepackage{amsmath}
\usepackage{amssymb}
\usepackage{amsthm}
\usepackage{mathtools}
\usepackage{empheq}
\usepackage{braket}
\usepackage{xcolor}
\usepackage[colorlinks=true, allcolors=blue]{hyperref}

\usepackage[margin=1.5in]{geometry}

\theoremstyle{plain}
\newtheorem{thm}{Theorem}[section]

\newtheorem{lem}[thm]{Lemma}
\newtheorem{conj}[thm]{Conjecture}

\theoremstyle{definition}
\newtheorem{defn}[thm]{Definition}

\newtheorem{rem}[thm]{Remark}
\numberwithin{equation}{section}

\newcommand{\la}{\lambda}
\newcommand{\ov}[1]{\overline{#1}}
\newcommand{\ZZ}{\mathbb{Z}}

\newcommand{\RR}{\mathbb{R}}
\newcommand{\CC}{\mathbb{C}}
\newcommand{\GL}{{\rm GL}}
\newcommand{\SL}{{\rm SL}}
\newcommand{\SO}{{\rm SO}}
\newcommand{\Sp}{{\rm Sp}}

\DeclareMathOperator{\cone}{cone}

\title{Cones and ping-pong in three dimensions}
\author{Gabriel Frieden, F\'elix G\'elinas, and \'Etienne Soucy}

\address{Universit\'e du Qu\'ebec \`a Montr\'eal, Montr\'eal, QC, Canada}
\email{gabriel.frieden@lacim.ca}
\email{gelinas.felix@courrier.uqam.ca}
\email{soucy.etienne@courrier.uqam.ca}

\begin{document}

\maketitle

\begin{abstract}
We study the hypergeometric group in $\GL_3(\CC)$ with parameters $\alpha = (\frac{1}{4}, \frac{1}{2}, \frac{3}{4})$ and $\beta = (0,0,0)$. We give a new proof that this group is isomorphic to the free product $\ZZ/4\ZZ * \ZZ/2\ZZ$ by exhibiting a ping-pong table. Our table is determined by a simplicial cone in $\RR^3$, and we prove that this is the unique simplicial cone (up to sign) for which our construction produces a valid ping-pong table.
\end{abstract}


\section{Introduction}
\label{sec:intro}

Beukers and Heckman \cite{BH} defined a \emph{hypergeometric group} to be a subgroup of $\GL_n(\CC)$ generated by three matrices $R,T,U$, such that $U^{-1}TR = I$, $U$ and $R$ have no shared eigenvalues, and $T - I$ is a rank one matrix. The name is due to the fact that these groups arise as monodromy groups of hypergeometric differential equations.

One of the main results of \cite{BH} says that the Zariski closure of a hypergeometric group is either a finite subgroup of $\GL_n(\CC)$, or one of the matrix groups $\SL_n(\CC), \SO_n(\CC), \Sp_n(\CC)$. If $H$ is a subgroup of $\GL_n(\ZZ)$ whose Zariski closure is $G(\CC)$ (where $G$ is a matrix group $\GL_n, \SL_n, \SO_n, \Sp_n$, etc.), $H$ is said to be \emph{arithmetic} if it has finite index in $G(\ZZ)$, and \emph{thin} otherwise. Arithmetic subgroups of $\GL_n(\ZZ)$ have long been a central object of study in number theory, but in recent years there has been increasing interest in thin subgroups (see \cite{Sarnak}).

The question of whether a given hypergeometric group is arithmetic or thin has been studied in \cite{CYY, V, SV, FMS, BravThomas, FF}, and is rather subtle. Fuchs, Meiri, and Sarnak showed that several infinite families of hypergeometric groups with closure $\SO_n(\CC)$ ($n$ odd) are thin \cite{FMS}. On the other hand, for hypergeometric groups with closure $\Sp_n(\CC)$, one infinite family is known to be arithmetic \cite{V}, but the only known thin examples are in $\Sp_4(\CC)$ \cite{SV, BravThomas, FF}.

In this paper, we are interested in a particular infinite family of hypergeometric groups. For $n \geq 2$, define $R_n, U_n, T_n \in \GL_n(\mathbb{C})$ by
\[
R_n = \begin{pmatrix}
0 & 0 & 0 &  & 0 & 0 & -1 \\
1 & 0 & 0 & \cdots & 0 & 0 & -1 \\
0 & 1 & 0 &  & 0 & 0 & -1 \\
& \vdots &  & \ddots &  & \vdots & \\
0 & 0 & 0 &  & 0 & 0 & -1 \\
0 & 0 & 0 & \cdots & 1 & 0 & -1 \\
0 & 0 & 0 & & 0 & 1 & -1
\end{pmatrix},
\qquad
U_n = \begin{pmatrix}
0 & 0 & 0 &  & 0 & 0 & \pm 1 \\
1 & 0 & 0 & \cdots & 0 & 0 & \mp n \\
0 & 1 & 0 &  & 0 & 0 & \pm \binom{n}{n-2} \\
& \vdots &  & \ddots &  & \vdots & \\
0 & 0 & 0 &  & 0 & 0 & \binom{n}{3} \\
0 & 0 & 0 & \cdots & 1 & 0 & -\binom{n}{2} \\
0 & 0 & 0 & & 0 & 1 & n
\end{pmatrix},
\]
and $T_n = U_nR_n^{-1}$, where the signs in the last column of $U_n$ alternate. Let $H_n$ be the hypergeometric group generated by $R_n, U_n, T_n$. The \emph{parameters} of $H_n$ are
\[
\alpha = \left(\frac{1}{n+1}, \frac{2}{n+1}, \ldots, \frac{n}{n+1}\right), \qquad \beta = (0, \ldots, 0).
\]
This means that the eigenvalues of $R_n$ and $U_n$ are $e^{\frac{2\pi i}{n+1}}, \ldots, e^{\frac{2\pi i n}{n+1}}$ and $1, \ldots, 1$, respectively. It follows from the criterion in \cite{BH} that $H_n$ has Zariski closure $\Sp_n(\CC)$ if $n$ is even, and $\SO_n(\CC)$ if $n$ is odd.\footnote{We remark that for odd $n$, the group $H_n = H_{2k+1}$ preserves a symmetric bilinear form of signature $(k+1,k)$. By contrast, \cite{FMS} studies the case of Lorentzian signature $(2k,1)$.} The group $H_n$ arises in algebraic geometry as the monodromy group of a well-studied family of degree $n$ hypersurfaces in $\mathbb{P}^{n-1}$ known as the \emph{Dwork family} (see, e.g., \cite{Katz}).

The group $H_n$ is known to be arithmetic when $n=2,3$ (see \cite{FMS}), and was shown by Brav and Thomas \cite{BravThomas} to be thin when $n=4$. According to Sarnak \cite{Sarnak}, it ``seems likely'' that $H_n$ is thin for all even $n \geq 4$. If this is true, it would provide the first examples of thin subgroups of $\Sp_n(\CC)$ for $n \geq 6$.

To show that $H_4$ is thin, Brav and Thomas used the ping-pong lemma to prove that $H_4$ is isomorphic to the free product $\ZZ/5\ZZ * \ZZ$. The following conjecture generalizes this result, and would imply that $H_n$ is thin for $n \geq 4$.

\begin{conj}
\label{conj:main}
If $n \geq 2$, then
\[
H_n = \langle R_n, T_n \rangle = \langle R_n \rangle * \langle T_n \rangle = \begin{cases}
\ZZ/(n+1)\ZZ * \ZZ & \text{ if } n \text { is even} \\
\ZZ/(n+1)\ZZ * \ZZ/2\ZZ & \text{ if } n \text{ is odd}.
\end{cases}
\]
\end{conj}

This paper undertakes a detailed study of Conjecture \ref{conj:main} in the case $n=3$. In \S \ref{sec:ping-pong}, we use the ping-pong lemma to give an elementary proof that $H_3 \cong \ZZ/4\ZZ * \ZZ/2\ZZ$. To apply the ping-pong lemma, one must define a ``ping-pong table'' in a set on which the group acts. In our case, we consider the natural action of $3 \times 3$ matrices on $\RR^3$, and our ping-pong table is determined by a simplicial cone $C$ in $\RR^3$. We prove in \S \ref{sec:uniqueness} that $C$ is (up to sign) the \emph{only} simplicial cone which gives rise to a ``valid'' ping-pong table via our construction.

In \S \ref{sec:projection}, we use a two-dimensional projection to illustrate the main ideas of the previous sections. Finally, in \S \ref{sec:2d 4d}, we compare our ping-pong table in the three-dimensional case with the (essentially unique) ping-pong table in the two-dimensional case, and with the more complicated ping-pong table of Brav and Thomas in the four-dimensional case. We hope that the juxtaposition of these three examples will inspire future work on Conjecture \ref{conj:main} in higher dimensions.

\begin{rem}
The $n=2$ and $n=3$ cases of Conjecture \ref{conj:main} can be obtained from classical results of Schwarz, Klein, and Clausen \cite{Schwarz,Klein,Clausen}. Indeed, Schwarz and Klein determined the structure of a large class of hypergeometric groups in $\GL_2$ (the so-called \emph{Schwarz triangle groups}), one of which is $H_2$. A result of Clausen implies that $H_3$ is the monodromy group of the symmetric square of one of the hypergeometric differential equations covered by the work of Schwarz and Klein (namely, the equation with parameters $\alpha = (1/8, 3/8)$ and $\beta = (0,0)$). It follows that $H_3$ is isomorphic to the Schwarz triangle group corresponding to these parameters. We refer the reader to \cite[\S\S 2.2, 3.2]{Heckman} for a nice account of this story.
\end{rem}

\subsection*{Acknowledgements}

We owe a great deal of thanks to Hugh Thomas, who suggested this problem, guided our work on it during the summer of 2021, and provided valuable feedback on an earlier version of this paper. We are grateful to Benjamin Dequ\^ene for his help throughout the summer. In addition, we acknowledge the open source software package Sage \cite{Sage}, which we used to carry out experiments and computations for this project.

GF was supported in part by the Canada Research Chairs program. FG and ES were supported by NSERC Discovery Grant RGPIN-2016-04872 and Undergraduate Summer Scholarships from the Institut des Sciences Math\'ematiques (ISM).


\section{A three-dimensional ping-pong table}
\label{sec:ping-pong}

\subsection{Cones}
Given vectors $v_1, \ldots, v_k \in \mathbb{R}^n$, define the \emph{open cone} $C$ generated by $v_1, \ldots, v_k$ to be the set of strictly positive linear combinations of the $v_i$. That is,
\[
C = \{a_1 v_1 + \ldots + a_k v_k \mid a_i \in \mathbb{R}_{> 0}\}.
\]
We will sometimes write $C = \cone(v_1, \ldots, v_k)$. Note that $C$ is unchanged if one of the generators $v_i$ is replaced by a positive scalar multiple $\la v_i, \la > 0$. The cone $C$ is said to be \emph{simplicial} if the generators $v_1, \ldots, v_k$ are linearly independent.

For a subset $S \subseteq \mathbb{R}^n$, we write $\ov{S}$ for the closure of $S$ (in the Euclidean topology). If $C = \cone(v_1, \ldots, v_k)$, then
\[
\ov{C} = \{a_1 v_1 + \ldots + a_k v_k \mid a_i \in \mathbb{R}_{\geq 0}\}.
\]
We call $\ov{C}$ the \emph{closed cone} generated by $v_1, \ldots, v_k$.

A subset $S \subseteq \mathbb{R}^n$ is \emph{convex} if for any two points $x,y \in S$, the line segment
\[
\{\la x + (1-\la) y \mid \la \in [0,1]\}
\]
connecting $x$ and $y$ is contained in $S$. It is easy to verify that cones (both open and closed) are convex.

\subsection{Free products and the ping-pong lemma}

Let $G_1, \ldots, G_d$ be subgroups of a group $G$. A \emph{word (in the elements of the $G_j$)} is a finite sequence $(x_1, \ldots, x_n)$, such that each $x_i$ is an element of at least one of the $G_j$. Each word gives rise to an element of $G$ by multiplication; that is, $(x_1, \ldots, x_n)$ gives rise to the element $g = x_1 \cdots x_n \in G$. In this case, we say that $(x_1, \ldots, x_n)$ is an \emph{expression} for $g$, or that $g$ can be \emph{expressed} as the word $(x_1, \ldots, x_n)$. The \emph{group generated by the subgroups $G_j$}, denoted $\langle G_1, \ldots, G_d \rangle$, is the subgroup of $G$ consisting of all elements that can be expressed as words in the elements of the $G_j$. The subgroups $G_j$ are said to \emph{generate} $G$ if $G = \langle G_1, \ldots, G_d \rangle$.

If $G_1, \ldots, G_d$ generate $G$, there are in general many expressions for each element of $G$ as a word in the elements of the $G_j$. We say that a word $(x_1, \ldots, x_n)$ is \emph{reduced} if none of the $x_i$ is the identity element, and for $i = 1, \ldots, n-1$, the elements $x_i$ and $x_{i+1}$ are not both contained in a single $G_j$. The idea is that identity elements can be removed from a word without changing the resulting element of $G$, and if $x_i, x_{i+1} \in G_j$, these two elements can be replaced by the single element $x_ix_{i+1} \in G_j$. By convention, the empty word gives rise to the identity element of $G$, and is considered to be reduced.

\begin{defn}
Let $G_1, \ldots, G_d$ be subgroups of a group $G$. The group $G$ is the \emph{free product} of the $G_j$ if each $g \in G$ has a unique reduced expression in the elements of the $G_j$. In this case, one writes
\[
G = G_1 * \cdots * G_d.
\]
\end{defn}

We encourage the reader to verify that if $G$ is the free product of $G_1, \ldots, G_d$, then $G_i \cap G_j = \{1\}$ for $i \neq j$. It may also be instructive to find a counterexample to the converse of this statement.

The following result, which is known as the \emph{ping-pong lemma}, is a standard tool for proving that two subgroups of a larger group generate a free product.

\begin{lem}[\cite{LyndonSchupp}]
Let $G,H$ be two non-trivial subgroups of a group $K$, such that at least one of $G$ and $H$ has more than two elements. Suppose $K$ acts on a set $S$, and there are two non-empty subsets $X,Y \subset S$ satisfying the following properties:
\begin{enumerate}
\item $X \cap Y = \emptyset$
\item If $g \in G \backslash \{1\}$ and $x \in X$, then $gx \in Y$
\item If $h \in H \backslash \{1\}$ and $y \in Y$, then $hy \in X$.
\end{enumerate}
Then the subgroup of $K$ generated by $G$ and $H$ is a free product; that is, $\langle G, H \rangle = G*H$.
\end{lem}

We will refer to the sets $X$ and $Y$ as a \emph{valid ping-pong table} (for $G$ and $H$) if they satisfy the hypotheses of the ping-pong lemma.

\subsection{A ping-pong table in $\mathbb{R}^3$}
We now consider the three-dimensional case of Conjecture \ref{conj:main}. Writing $R = R_3$, $U = U_3$, and $T = T_3$, we have
\begin{align*}
  &R= \begin{pmatrix*}[r]
    0 & 0 & -1\\
    1 & 0 & -1\\
    0 & 1 & -1
  \end{pmatrix*}
  &&U= \begin{pmatrix*}[r]
    0 & 0 & 1\\
    1 & 0 & -3\\
    0 & 1 & 3
  \end{pmatrix*}
  &&T= \begin{pmatrix*}[r]
    -1 & 0 & 0\\
    2 & 1 & 0\\
    -4 & 0 & 1
  \end{pmatrix*}.
\end{align*}
Note that $R^4 = T^2 = I$.

\begin{thm}
\label{thm:main}
The subgroup of $\GL_3(\RR)$ generated by $R$ and $T$ is the free product of $\braket{R}$ and $\braket{T}$; that is,
\[
\braket{R,T} = \ZZ/4\ZZ * \ZZ/2\ZZ.
\]
\end{thm}

\begin{proof}
The group $GL_3(\RR)$ acts on $\RR^3$ by matrix multiplication. We will find disjoint subsets $X,Y \subset \RR^3$ such that all elements of $X$ are sent to $Y$ by $R,R^2,$ and $R^3$, and all elements of $Y$ are sent to $X$ by $T$, which will allow us to conclude that $\braket{R, T} \cong \braket{R} * \braket{T}$ by the ping-pong lemma.

Let $C$ be the open cone generated by the vectors
\begin{align*}
  &u = (1, -2, 1), & &v = (1, 0, 3), & &w = (0, -1, 1),
\end{align*}
that is, $C = \{au + bv + cw \mid a \in \RR_{> 0}\}$. Define
\[
X = C \cup -C, \qquad\qquad Y = RX \cup R^2X \cup R^3X.
\]
It is immediately clear from this definition that each non-identity element of $\braket{R}$ maps $X$ into $Y$, so hypothesis (2) of the ping-pong lemma is satisfied.

Now suppose there is a point $p = au + bv + cw \in X \cap Y$. Since $p \in X$, the coefficients
$a,b,c$ are all nonzero and of the same sign. Since $p \in Y$, there exists a point $q = xu + yv + zw \in X$ (so again $x,y,z$ are nonzero and of the same sign), such that $R, R^2$ or $R^3$
maps $q$ to $p$. Explicitly, we have
\begin{align*}
  &q = (x + y, -2x - z, x + 3y + z)\\
  &Rq = (-x -3y - z, -2y - z, -3x - 3y - 2z)\\
  &R^2q = (3x + 3y + 2z, 2x + z, 3x + y + z)\\
  &R^3q = (-3x - y - z, 2y + z, -x - y)\\
  &p = (a + b, -2a - c, a + 3b + c).
\end{align*}
This gives us three systems $p = R^iq$ which solve to
\begin{align*}
&p = Rq && \implies &&a = -y &&b = -x - 2y - z &&c = 4y + z\\
&p = R^2q && \implies &&a = x + 2y + z &&b = 2x + y + z &&c= -4x - 4y - 3z\\
&p = R^3q && \implies &&a = -2x - y - z &&b = -x &&c = 4x + z.
\end{align*}
Remembering that the triples $(x,y,z)$ and $(a,b,c)$ must be nonzero and either
all positive or all negative, we obtain a contradiction in each case:
\begin{itemize}
  \item In the first case, if $x,y,z$ are positive, then $a = -y$ is negative, but
  $c = 4y + z$ is positive, and vice versa in the negative case.
  \item In the second case, again $a = x + 2y + z$ and $c = -4x - 4y -3z$ cannot
  have the same sign if $x,y,z$ have the same sign.
  \item The same goes in the third case for $a = -2x - y - z$ and $c = 4x + z$.
\end{itemize}
These contradictions prove that $X$ and $Y$ are indeed disjoint.

We will now verify that $T$ sends $Y$ into $X$ using a similar argument. As before, let $q = xu + yv + zw$ be a point in $X$. If we apply $T$ to $Rq, R^2q, R^3q$, we get
\begin{align*}
  &TRq = (x+3y+z,-2x-8y-3z,x+9y+2z)\\
  &TR^2q = (-3x-3y-2z,8x+6y+5z,-9x-11y-7z)\\
  &TR^3q = (3x+y+z,-6x-z,11x+3y+4z).
\end{align*}
This time solving the systems $p = TR^iq$ (where $p = au + bv + cw$) nets us:
\begin{align*}
  &p = TRq && \implies &&a=x+2y+z &&b=y &&c=4y+z\\
  &p = TR^2q && \implies &&a=-2x-y-z &&b=-x-2y-z &&c=-4x-4y-3z\\
  &p = TR^3q && \implies &&a=x &&b=2x+y+z &&c=4x+z
\end{align*}
In this case we see that the signs of $a,b,c$ all properly match,
which confirms that $T$ does send $Y$ into $X$, completing the proof.
\end{proof}

\subsection{Matrix logarithms}
\label{sec:logs}

At this point, the reader may be wondering how we arrived at the definition of the cone $C$. The explanation requires an examination of the linear maps $TR$ and $TR^{-1}$, and their logarithms. In addition to motivating the choice of generators $u,v,$ and $w$, the formulas derived below play an essential role in the proof of the uniqueness of $C$ in the next section.

The matrix $U = TR$ has Jordan form
\[
\begin{pmatrix}
1 & 1 & 0 \\
0 & 1 & 1 \\
0 & 0 & 1
\end{pmatrix}.
\]
This means that 1 is the only eigenvalue of $TR$, and the corresponding eigenspace is 1-dimensional. The vector $u = (1,-2,1)$ spans this eigenspace. The matrix
\[
R(TR)^{-1}R^{-1} = TR^{-1}
\]
has the same Jordan form as $TR$, and its 1-dimensional eigenspace is spanned by $v = (1,0,3)$.

By the hypotheses of the ping-pong lemma, any positive integer power of the linear transformations $TR$ and $TR^{-1}$ must map $X$ to itself. To understand the powers of these matrices, we use the Taylor expansions of $\log$ and $\exp$, which allow us to define
\[
(TR)^t = \exp(t \log(TR)) \qquad \text{ and } \qquad (TR^{-1})^t = \exp(t \log(TR^{-1}))
\]
for all $t \in \RR$. For $TR$, we compute
\begin{align*}
\log(TR) &= (TR-I) - \dfrac{(TR-I)^2}{2} + \dfrac{(TR-I)^3}{3} - \cdots \\
&=
\begin{pmatrix}
-1 & 0 & 1 \\
1 & -1 & -3 \\
0 & 1 & 2
\end{pmatrix}
- \dfrac{1}{2}
\begin{pmatrix}
1 & 1 & 1 \\
-2 & -2 & -2 \\
1 & 1 & 1
\end{pmatrix}
+ 0 - \cdots =
\begin{pmatrix}
-\frac32 & -\frac12 & \frac12 \\
2 & 0 & -2 \\
-\frac12 & \frac12 & \frac32
\end{pmatrix},
\end{align*}
and then
\begin{multline}
\label{eq:TR^t}
(TR)^t = \exp(t \log(TR)) = I + t \log(TR) + \dfrac{t^2}{2!} \log(TR)^2 + \cdots \\
=
\begin{pmatrix}1 & 0 & 0 \\ 0 & 1 & 0 \\ 0 & 0 & 1
\end{pmatrix}+t\begin{pmatrix}-\frac32 & -\frac12 & \frac12 \\ 2 & 0 & -2 \\ -\frac12 & \frac12 & \frac32
\end{pmatrix}+t^2\begin{pmatrix}\frac12 & \frac12 & \frac12 \\ -1 & -1 & -1 \\ \frac12 & \frac12 & \frac12
\end{pmatrix}.
\end{multline}
Similarly, we compute
\begin{equation}
\label{eq:TR^-1^t}
(TR^{-1})^t = \begin{pmatrix}1 & 0 & 0 \\ 0 & 1 & 0 \\ 0 & 0 & 1
\end{pmatrix}+t\begin{pmatrix}-\frac32 & -\frac12 & \frac12 \\ -3 & 1 & 1 \\ -\frac32 & -\frac52 & \frac12
\end{pmatrix}+t^2\begin{pmatrix} \frac32 & -\frac12 & -\frac12 \\ 0 & 0 & 0 \\ \frac92 & -\frac32 & -\frac32 \end{pmatrix}.
\end{equation}

Let $P = \log(TR)$ and $Q = \log(TR^{-1})$ (these are the coefficients of $t$ in \eqref{eq:TR^t} and \eqref{eq:TR^-1^t}, respectively). As the reader may easily verify, both $P$ and $Q$ have rank two, and their column spaces intersect in the line spanned by $w = (0,-1,1)$. It is perhaps not clear why this intersection should be useful in defining a ping-pong table. In \S \ref{sec:uniqueness 2d}, we consider a two-dimensional projection that clearly illustrates the significance of this intersection.


\section{Uniqueness of the cone $C$}
\label{sec:uniqueness}

Let $C'$ be the open cone generated by three linearly independent vectors $u',v',w' \in \mathbb{\RR}^3$, and define
\[
X = C' \cup - C', \qquad\qquad Y = RX \cup R^2X \cup R^3X.
\]
The goal of this section is to prove the following uniqueness theorem.

\begin{thm}
\label{thm:uniqueness}
If $X$ and $Y$ form a valid ping-pong table for $\langle R \rangle$ and $\langle T \rangle$, then $C' = C$ or $C' = -C$, where $C = \cone(u,v,w)$ is the cone defined in the previous section.
\end{thm}

The proof consists of two steps, the first of which is carried out in the following lemma.

\begin{lem}
\label{lem:uv}
Suppose $X$ and $Y$ form a valid ping-pong table for $\langle R \rangle$ and $\langle T \rangle$.
\begin{enumerate}
\item[(a)] Let $M = (TR^j)^t$ for fixed $j \in \{1,2,3\}$ and $t \in \ZZ_{> 0}$. Either
\[
M(\ov{C'}) \subseteq \ov{C'} \qquad \text{ or } \qquad M(\ov{C'}) \subseteq -\ov{C'}.
\]
\item[(b)] The lines spanned by $u$ and $v$ are contained in $\ov{X}$.
\item[(c)] Two of the generators of $C'$ are $u$ and $v$ (or $-u$ and $-v$).
\end{enumerate}
\end{lem}

\begin{proof}
The hypotheses of the ping-pong lemma imply that $M$ maps $X$ into $X$. Since linear transformations are continuous, this implies that $M$ maps $\ov{X}$ into $\ov{X}$. Suppose there are non-zero vectors $q_1,q_2 \in \ov{C'}$ such that $M(q_1) \in \ov{C'}$ and $M(q_2) \in -\ov{C'}$. Linear transformations map line segments to line segments, so the convexity of $\ov{C'}$ implies that the line segment from $M(q_1)$ to $M(q_2)$ is contained in $\ov{X} = \ov{C'} \cup -\ov{C'}$. This can only happen if the line segment connecting $M(q_1)$ and $M(q_2)$ passes through the origin, that is, if $M(q_1) = -\la M(q_2)$ for some $\la > 0$. Since $M$ is invertible, this would imply that $q_1 = - \la q_2$, so $q_1, q_2 \in \ov{C'} \cap -\ov{C'} = \{0\}$, a contradiction. This proves (a).

To prove part (b), we will show that for any nonzero vector $q = (x,y,z) \in \RR^3$, the vectors $(TR)^t(q)$ approach the line generated by $u$ as $t$ approaches infinity, and the vectors $(TR^{-1})^t(q)$ approach the line generated by $v$. By \eqref{eq:TR^t}, we have
\begin{equation}
\label{eq:TR^t(s)}
(TR)^{t}(q)= \begin{pmatrix}
\frac12(x+y+z)t^2 + \frac12(-3x-y+z)t + x \smallskip \\
-(x+y+z)t^2 + 2(x-z)t + y \smallskip \\
\frac12(x+y+z)t^2 + \frac12(-x+y+3z)t + z  \end{pmatrix}.
\end{equation}
For a nonzero vector $a$, let $\widehat{a}$ denote the normalization of $a$ (i.e., the vector $a$ divided by its Euclidean norm). Using the fact that $\lim_{t \to \infty} \widehat{(TR)^t(q)}$ depends only on the coefficients of the highest power of $t$ appearing in $(TR)^t(q)$, we find that
\begin{equation}
\label{eq:U^t}
\lim_{t \to \infty} \widehat{(TR)^t(q)} = \begin{cases}
\dfrac{x+y+z}{|x+y+z|} \widehat{u} & \text{ if } x+y+z \neq 0 \medskip \\
\dfrac{z-x}{|z-x|} \widehat{u} & \text{ if } x+y+z = 0 \text{ and } x \neq z \medskip \\
\dfrac{x}{|x|} \widehat{u} & \text{ if } x+y+z = 0 \text{ and } x = z
\end{cases}.
\end{equation}
In all cases, the normalization of $(TR)^t(q)$ approaches $\pm \widehat{u}$, one of the two unit eigenvectors of $TR$. Similarly, using \eqref{eq:TR^-1^t}, we find that
\begin{equation}
\label{eq:U'^t}
\lim_{t\to\infty}\widehat{(TR^{-1})^t(q)} = \begin{cases}
\dfrac{3x-y-z}{|3x-y-z|} \widehat{v} & \text{ if } 3x-y-z \neq 0 \medskip \\
\dfrac{-y}{|y|} \widehat{v} & \text{ if } 3x-y-z = 0 \text{ and } y \neq 0 \medskip \\
\dfrac{x}{|x|} \widehat{v} & \text{ if } 3x-y-z = 0 \text{ and } y = 0
\end{cases},
\end{equation}
so in all cases the normalization of $(TR^{-1})^t(q)$ approaches $\pm \widehat{v}$, one of the two unit eigenvectors of $TR^{-1}$.

If $q \in X$, then as observed in the proof of part (a), $(TR)^t(q)$ and $(TR^{-1})^t(q)$ must be in $X$ for any positive integer $t$. Thus, since $X$ is closed under scalar multiplication (and non-empty), the previous calculations tell us that each point on the lines spanned by $u$ and $v$ is a limit point of a sequence of points in $X$, so these lines are in the closure of $X$. This proves (b).

It remains to prove (c). By part (b), we may assume that $u$ is contained in $\ov{C'}$ (possibly after replacing $C'$ with $-C'$). Suppose that $u$ is not a generator of $C'$. This means that $u$ is contained in the interior of $\ov{C'}$, or in the interior of a face of $\ov{C'}$. In either case, we can find a vector $q = (x,y,z)$ which is not a scalar multiple of $u$, such that the line segment
\[
\{u + \la q \mid |\la| \leq \epsilon\}
\]
is contained in $\ov{C'}$ for sufficiently small $\epsilon > 0$. All points $(a,b,c)$ which satisfy both $a+b+c = 0$ and $c-a = 0$ are on the line spanned by $u = (1,-2,1)$, so we must have $x+y+z \neq 0$ or $z-x \neq 0$. We may assume that $x+y+z > 0$, or that $x+y+z = 0$ and $z-x > 0$. By \eqref{eq:U^t}, the sequence $(TR)^t(u + \la q)$ approaches the ray generated by $u$ if $\la \geq 0$, and the ray generated by $-u$ if $\la < 0$. This contradicts part (a).

A similar argument using \eqref{eq:U'^t} shows that $v$ must be a generator of $C'$ or $-C'$. To see that $v$ must in fact be a generator of $C'$, note that $(TR)^t(u) = u$ for all $t$, and $(TR)^t(v) = (TR)^t(1,0,3)$ approaches the ray generated by $u$ by \eqref{eq:U^t}. Now part (a) guarantees that $v \not \in -\ov{C'}$.
\end{proof}

\begin{rem}
The proof of part (b) works for any valid ping-pong table in which $X$ is closed under scalar multiplication.
\end{rem}

\begin{proof}[Proof of Theorem \ref{thm:uniqueness}]
By Lemma \ref{lem:uv}(c), we may assume (possibly after replacing $C'$ with $-C'$) that two of the generators of $C'$ are $u$ and $v$. Suppose $C' = \cone(u,v,w')$, where
\[
w' = \la u + \mu v + \eta w = \begin{pmatrix}
\la + \mu \\
-2\la -\eta \\
\la + 3\mu + \eta
\end{pmatrix}
\]
for some $\la,\mu,\eta \in \mathbb{R}$. Since $u,v,w'$ are assumed to be linearly independent, we must have $\eta \neq 0$. We first show that $\eta > 0$.

Applying \eqref{eq:TR^t(s)} to $v = (1,0,3)$, we obtain
\[
(TR)^t(v) = \begin{pmatrix}
2t^2 + 1 \\
-4t^2 - 4t \\
2t^2 + 4t + 3
\end{pmatrix}.
\]
Solving a system of linear equations, we find that $(TR)^t(v) = au + bv + cw'$, where
\[
a = 2t^2 - \frac{4\la}{\eta}t, \qquad b = 1 - \frac{4\mu}{\eta}t, \qquad c = \frac{4}{\eta}t.
\]
Since $v \in \ov{X}$, the hypotheses of the ping-pong lemma require that $(TR)^t(v)$ be in $\ov{X}$ for all $t \in \mathbb{Z}_{> 0}$. This means that for such $t$, we must have $a,b,c \geq 0$ or $a,b,c \leq 0$. For large $t$, $a$ is positive and $c$ has the same sign as $\eta$. This shows that $\eta$ must be positive, as claimed.

Scaling $w'$ by a positive constant does not change $C'$, so we may assume that $w' = \la u + \mu v + w$. We will now show that $TR(X) \not \subseteq X$ if $\mu \neq 0$, and $TR^{-1}(X) \not \subseteq X$ if $\la \neq 0$. Suppose $x,y,z > 0$, so that $q = xu + yv + zw'$ is in $C'$. Solving a system of linear equations, we find that $TR(q) = au + bv + cw'$, where
\begin{align*}
a &= x + (2-4\la)y + (1+2\mu-4\la\mu)z \\
b &= (1-4\mu)y - 4\mu^2 z \\
c &= 4y + (1+4\mu)z.
\end{align*}
The crucial feature of these formulas is the presence of $\mu^2$ in the equation for $b$. This means that if $\mu \neq 0$, then by choosing $z$ sufficiently large, we can make $b$ negative. But for any fixed choice of $z$, we can make $a$ positive by choosing $x$ sufficiently large. This shows that there is a choice of $x,y,z > 0$ such that $a$ and $b$ do not have the same sign, contradicting the assumption that $TR$ maps $X$ to itself. We conclude that $\mu = 0$.

Next, we compute $TR^{-1}(q) = a'u + b'v + c'w'$, where
\begin{align*}
a' &= (1-4\la)x - 4\la^2 z \\
b' &= y + (2-4\mu)x + (1+2\la-4\la\mu)z \\
c' &= 4x + (1+4\la)z.
\end{align*}
If $\la \neq 0$, we can make $a'$ negative by choosing $z$ sufficiently large, and then we can make $b'$ positive by choosing $y$ sufficiently large. This contradicts the assumption that $TR^{-1}$ maps $X$ to itself, so we must have $\la = 0$. We conclude that $w'$ is a positive scalar multiple of $w$.
\end{proof}


\section{Two-dimensional projection}
\label{sec:projection}

\subsection{Definition of the projection}
\label{sec:proj def}

In order to better understand the algebraic arguments in the previous sections, it is useful to project from $\RR^3$ to a plane, where we can more easily visualize what is going on. Given a linear functional $\phi \colon \RR^3 \rightarrow \RR$, we can send a vector $s \in \RR^3$ to $s/\phi(s)$, provided $\phi(s) \neq 0$. Since $\phi$ is linear, $\phi(s/\phi(s)) = \phi(s)/\phi(s) = 1$. Thus, the map $\rho \colon s \mapsto s/\phi(s)$ amounts to projecting $s$ onto the plane $P = \{s \mid \phi(s)=1\}$.

We will use the projection $\rho$ determined by the linear functional
\[
\phi(x,y,z) = x-y+z.
\]
This choice of $\phi$ satisfies $\phi(u), \phi(v), \phi(w) > 0$, so the cone generated by $u,v,w$ projects to a triangle in the plane $P$. We need to choose a system of coordinates on $P$. The fundamental theorem of affine geometry tells us that for any three points $p,q,r \in \mathbb{R}^2$ which are not collinear, there is a unique affine transformation from $P$ to $\mathbb{R}^2$ sending $\rho(u), \rho(v), \rho(w)$ to $p,q,r$. For simplicity, we choose
\[
p = (0,1), \qquad q = (1,0), \qquad r = (1,1),
\]
which leads to the projection map
\begin{equation}
\label{eq:projection}
\rho(x,y,z) = \left(\frac{-2(x-z)}{x-y+z}, \frac{-2y}{x-y+z}\right).
\end{equation}
Applying $\rho$ to $X$ and $Y$, we obtain Figure \ref{fig:projection_uvw}, which illustrates the fact that $X$ and $Y$ define a valid ping-pong table.

\begin{figure}[h]
\centering \includegraphics[scale=3]{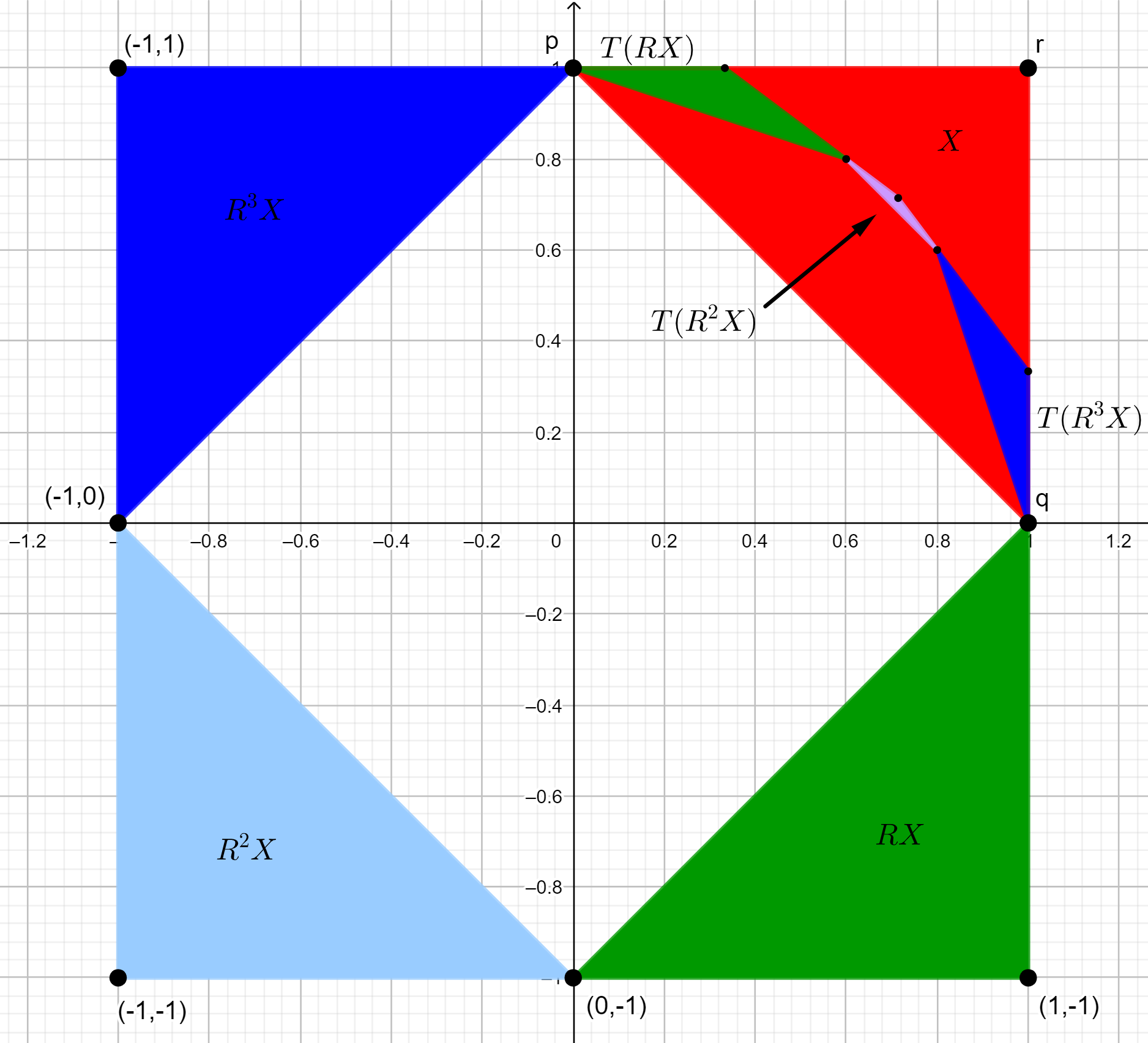}
\caption{The large triangle in the first quadrant (colored red) is the projection of $X = \pm C$. The three large triangles in the other quadrants are the projections of $RX, R^2X$, and $R^3X$. The smaller triangles in the first quadrant are the projections of $T(RX), T(R^2X)$, and $T(R^3X)$.}
\label{fig:projection_uvw}
\end{figure}

We now express the maps $R$ and $T$ in terms of the coordinates $(a,b)$ on $\mathbb{R}^2$. The point $(a,b)$ is the image of a line in $\mathbb{R}^3$, and a straightforward computation shows that the line which maps to $(a,b)$ is spanned by the vector $(x,y,z)$, where
\begin{align}
  x = -\frac{a}{4} - \frac{b}{4} + \frac{1}{2} \qquad\qquad y = -\frac{b}{2} \qquad\qquad z =  \frac{a}{4} - \frac{b}{4} + \frac{1}{2}.
\end{align}
If we apply $R$ and $T$ to $(x,y,z)$ and then apply $\rho$, we obtain the following formulas for the actions of $R$ and $T$ on $\mathbb{R}^2$:
\begin{align}
  \label{eq:R} R(a,b) &= (b, -a),\\
  \label{eq:T} T(a,b) &= \left(\frac{2a + b - 2}{2a + 2b - 3}, \frac{a + 2b - 2}{2a + 2b - 3}\right).
\end{align}
In particular, $R$ is rotation by 90 degrees (clockwise). This explains why the projection of $Y$ consists of the rotations of $X$ (the red triangle) in Figure \ref{fig:projection_uvw}.

\subsection{Uniqueness revisited}
\label{sec:uniqueness 2d}

Using the projection $\rho$, we can give a more visual explanation of the uniqueness of the cone $C$. The following argument is similar in spirit to the proof of uniqueness given in \S \ref{sec:uniqueness}, although it does not exactly correspond to the steps of that proof.

\begin{figure}
\centering \includegraphics[scale=0.1]{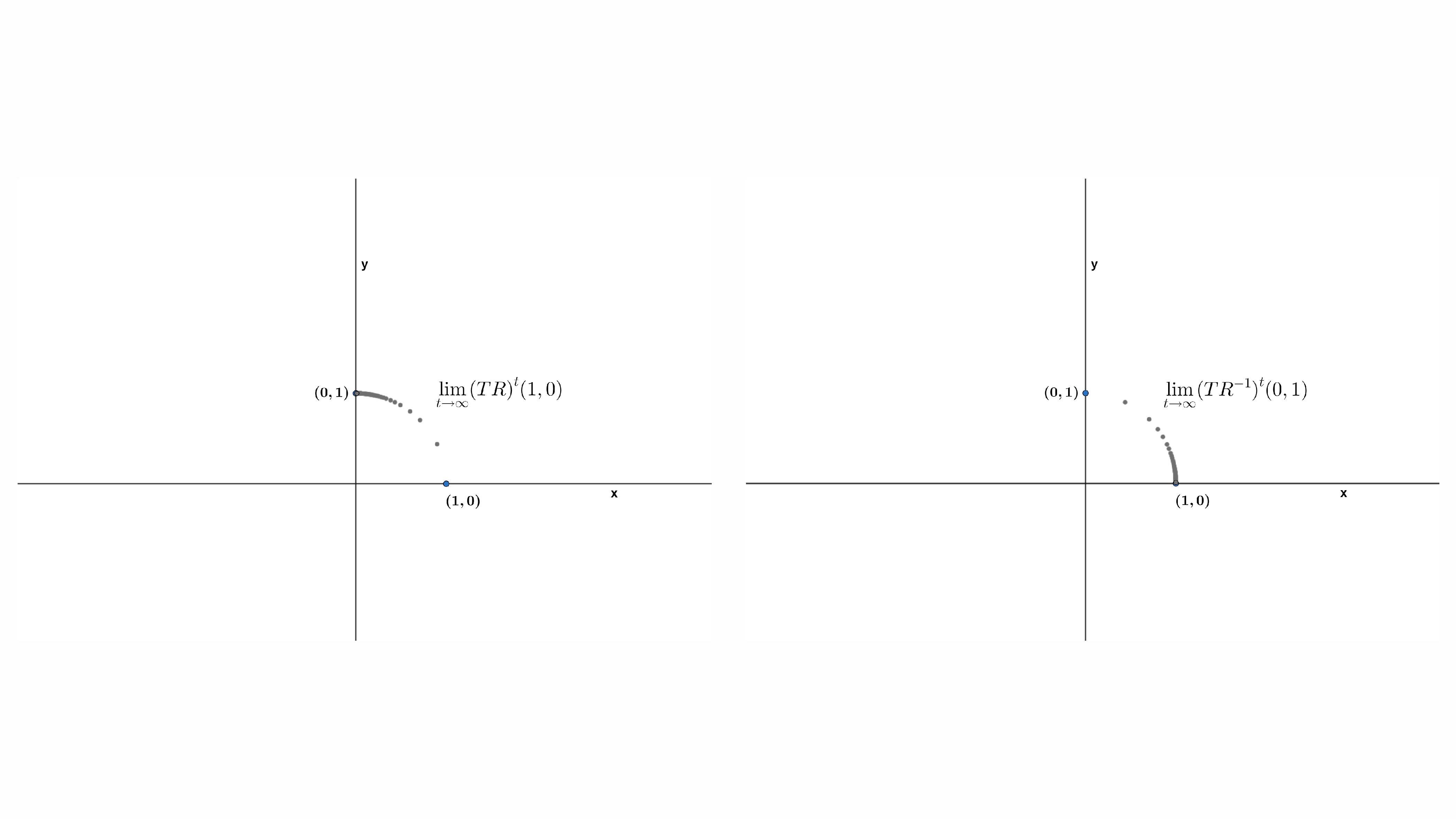}
\caption{The points $(TR)^t(v)$ (left) and $(TR^{-1})^t(u)$ (right) in the 2D projection.}
\label{fig:2D limit}
\end{figure}

By construction, the eigenvectors $u$ and $v$ project to the points $p = (0,1)$ and $q = (1,0)$. The sequences of points $(TR)^t(1,0)$ and $(TR^{-1})^t(0,1)$ (for $t \in \ZZ_{>0}$) are shown in Figure \ref{fig:2D limit}. These figures suggest that the first sequence approaches $(0,1)$ along the unit circle, and the second sequence approaches $(1,0)$ along the unit circle; this is verified in \S \ref{sec:circle}. The tangent line to the curve $(TR)^t(1,0)$ becomes horizontal as $t \rightarrow \infty$, and the tangent line to $(TR^{-1})^t(0,1)$ becomes vertical as $t \rightarrow \infty$. If the triangle formed by $(0,1), (1,0)$, and a third point $s$ determines a valid ping-pong table, then the triangle must contain the intersection of these two tangent lines, which is the point $(1,1)$. Thus, $s$ must lie in one of the closed cones $\ov{Q_1}$ or $\ov{Q_2}$ defined by
\[
\ov{Q_1}= {\mathbb{R}_{\geq 0}}\begin{pmatrix}0 \\ 1 \end{pmatrix} + {\mathbb{R}_{\geq 0}}\begin{pmatrix}1 \\ 1 \end{pmatrix} + \begin{pmatrix}1 \\ 1 \end{pmatrix}
\]
and
\[
\ov{Q_2}={\mathbb{R}}_{\geq 0}\begin{pmatrix}1 \\ 0 \end{pmatrix}+{\mathbb{R}}_{\geq 0}\begin{pmatrix}1 \\ 1 \end{pmatrix}+\begin{pmatrix}1 \\ 1
\end{pmatrix}.
\]
These cones are shown in Figure \ref{fig:Q1 Q2}.

Let
\[
s=\lambda_1\begin{pmatrix}0 \\ 1 \end{pmatrix}+\lambda_2\begin{pmatrix}1 \\ 1 \end{pmatrix} +\begin{pmatrix}1 \\ 1 \end{pmatrix}
= \begin{pmatrix}\lambda_2+1 \\ \lambda_1+\lambda_2+1\end{pmatrix}
\]
be a point in $\ov{Q_1} \setminus (1,1)$. This means that $\lambda_1, \lambda_2 \in {\mathbb{R}}_{\geq 0}$, and $\lambda_1,\lambda_2$ are not both zero. If $X',Y'$ are a valid ping-pong table, then every point of $Y'$ must be sent into $X'$ by $T$. By continuity, this implies that $T$ must send every point in $\ov{Y'}$ to $\ov{X'}$. The point $Rs$ is in $\ov{Y'}$, but we will show that $TRs$ is not in $\ov{X'}$.

\begin{figure}
\centering \includegraphics[width=10cm]{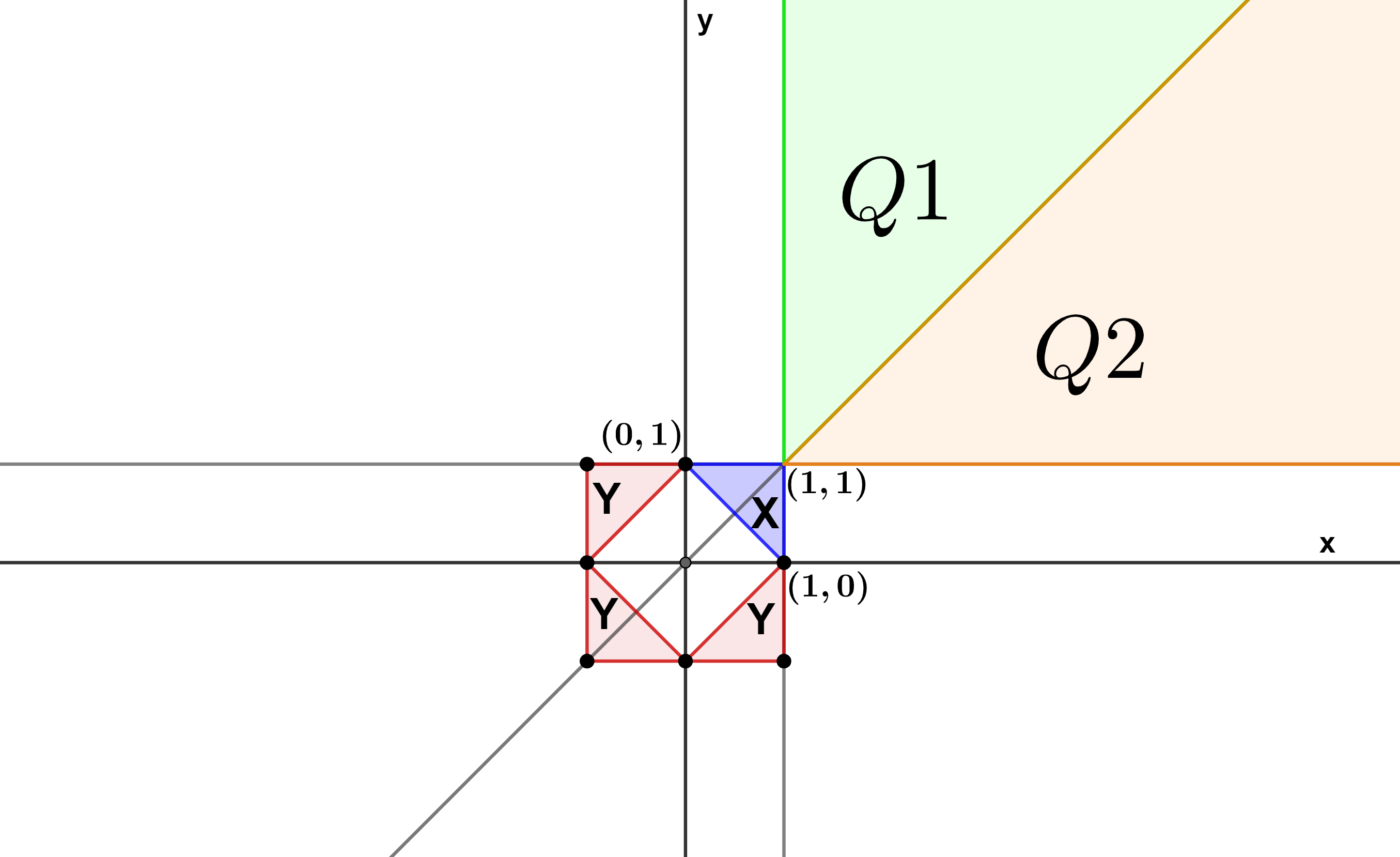}
\caption{The cones $Q_1$ and $Q_2$.}
\label{fig:Q1 Q2}
\end{figure}

The closed triangle $\ov{X'}$ is the intersection of the closed cones $\ov{X_1'}$ and $\ov{X_2'}$ defined by
\begin{align*}
\ov{X_1'} &= {\mathbb{R}}_{\geq 0}\left(s - \begin{pmatrix} 0 \\ 1 \end{pmatrix} \right) + {\mathbb{R}}_{\geq 0}\begin{pmatrix}1 \\ -1 \end{pmatrix}+\begin{pmatrix}0 \\ 1
\end{pmatrix}, \\
\ov{X_2'} &= {\mathbb{R}}_{\geq 0}\left(s - \begin{pmatrix} 1 \\ 0 \end{pmatrix} \right) +{\mathbb{R}}_{\geq 0}\begin{pmatrix}-1 \\ 1 \end{pmatrix} +\begin{pmatrix}1 \\ 0
\end{pmatrix}.
\end{align*}
These cones are illustrated in Figure \ref{fig:X_1', X_2'}. Suppose $TRs \in \ov{X_1'} \cap \ov{X_2'}$. This means there are $a,b,c,d \geq 0$ such that
\[
TRs = a\begin{pmatrix}\lambda_2 + 1 \\ \lambda_1 + \lambda_2
\end{pmatrix}+b\begin{pmatrix}1 \\ -1 \end{pmatrix}+\begin{pmatrix}0 \\ 1
\end{pmatrix} = c\begin{pmatrix}\lambda_2 \\ \lambda_1 + \lambda_2 + 1
\end{pmatrix}+d\begin{pmatrix}-1 \\ 1 \end{pmatrix}+\begin{pmatrix}1 \\ 0
\end{pmatrix}.
\]

\begin{figure}
\centering \includegraphics[width=10cm]{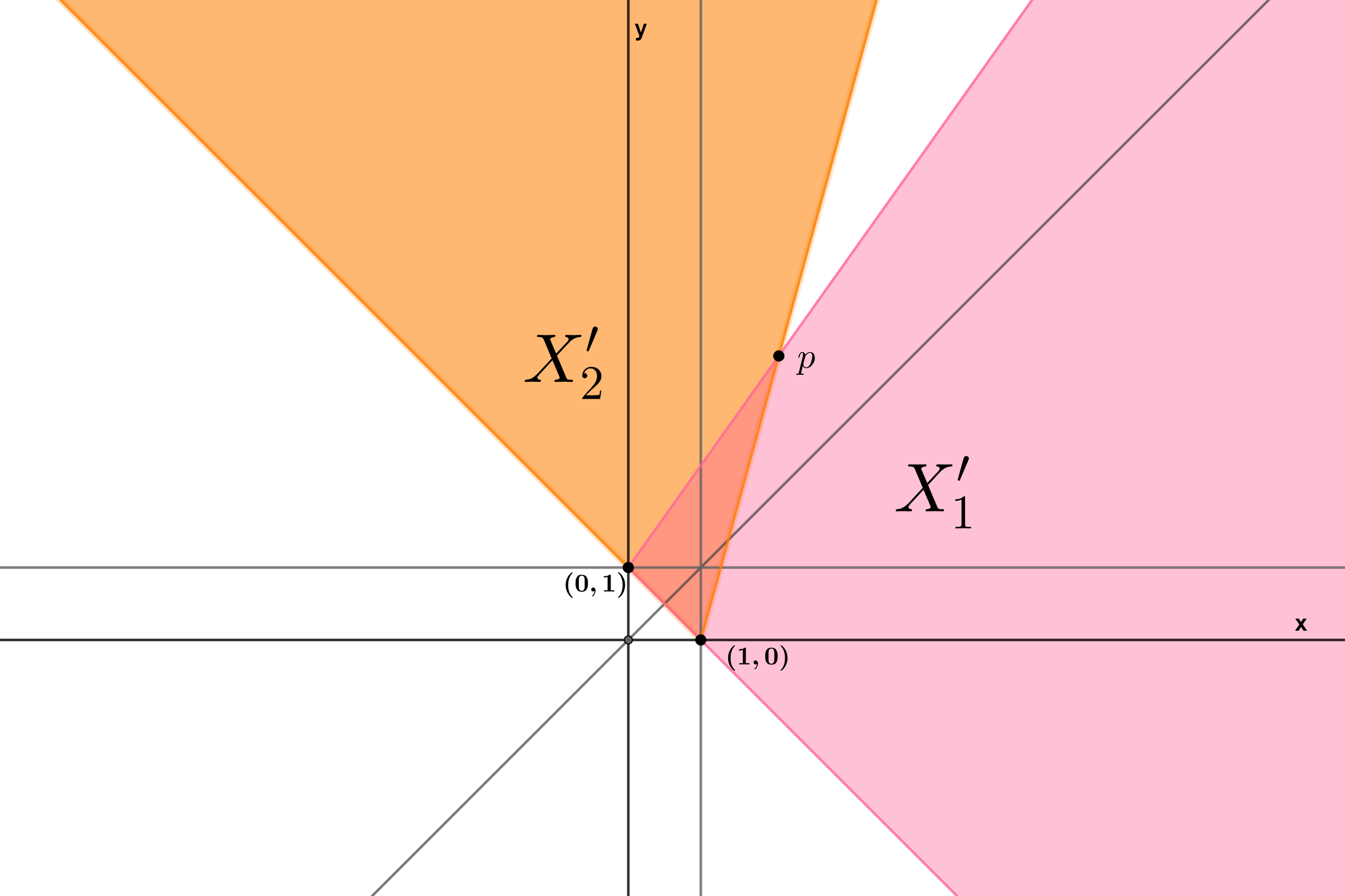}
\caption{The cones $X'_1$ and $X'_2$, whose intersection is $X'$.}
\label{fig:X_1', X_2'}
\end{figure}

Using \eqref{eq:R} and \eqref{eq:T}, we compute
\[
TRs=T\begin{pmatrix}\lambda_1+\lambda_2+1 \\-\lambda_2-1
\end{pmatrix}=
\begin{pmatrix}\dfrac{2\lambda_1+\lambda_2-1}{2\lambda_1-3} \medskip \\
\dfrac{\lambda_1-\lambda_2-3}{2\lambda_1-3}\end{pmatrix},
\]
so $a,b,c,d$ must be a solution to the system of linear equations
\[
\begin{matrix}
a(\lambda_2+1)+b=\dfrac{2\lambda_1+\lambda_2-1}{2\lambda_1-3} \quad & \quad
c\lambda_2-d+1=\dfrac{2\lambda_1+\lambda_2-1}{2\lambda_1-3} \medskip \\
a(\lambda_1+\lambda_2)-b+1=\dfrac{\lambda_1-\lambda_2-3}{2\lambda_1-3} \quad & \quad
c(\lambda_1+\lambda_2+1)+d=\dfrac{\lambda_1-\lambda_2-3}{2\lambda_1-3}.
\end{matrix}
\]
This system of equations has the unique solution
\[
a=\dfrac{\lambda_1-1}{\theta}, \quad b=\dfrac{2(\lambda_1^2+2\lambda_1\lambda_2+\lambda_2^2)}{\theta}, \quad c=\dfrac{\lambda_1-1}{\theta}, \quad d=\dfrac{-2(\lambda_2^2+\lambda_1+3\lambda_2+1)}{\theta},
\]
where
\[
\theta= (2\lambda_1 - 3)(\lambda_1 + 2\lambda_2 + 1).
\]
Since $\lambda_1,\lambda_2 \in \mathbb{R}_{\geq 0}$ and at least one of $\lambda_1$ and $\lambda_2$ is nonzero, $b$ and $d$ have opposite signs, contradicting that both are $\geq 0$. Thus, we cannot express $TRs$ both as a non-negative linear combination of generators of $\ov{X'_1}$, and as a non-negative linear combination of generators of $\ov{X'_2}$, so $TRs \not \in \ov{X'}$.

Now suppose $s \in \ov{Q_2} \setminus (1,1)$. Let $F \colon (a,b) \mapsto (b,a)$ be reflection over the line $x = y$. It is clear from \eqref{eq:R} and \eqref{eq:T} that $FRF = R^{-1}$ and $FTF = T$. Since $\ov{Q_2}$ is the reflection of $\ov{Q_1}$ over the line $x = y$, we conclude from the previous argument that $TR^{-1}s \not \in \ov{X'}$, so again $X'$ and $Y'$ are not a valid ping-pong table.

\subsection{A smaller ping-pong table}
\label{sec:circle}

We have shown that $C$ is the only simplicial cone that can be used to define a valid ping-pong table. If we drop the requirement that $C$ be a simplicial cone, however, then we have additional possibilities. As Figure \ref{fig:projection_uvw} illustrates, $T$ maps the triangles $RX, R^2X,$ and $R^3X$ to three smaller triangles inside $X$. We can therefore obtain a smaller ping-pong table by defining $X$ to be the union of these three triangles, and $Y$ to be the union of the images of these triangles under $R, R^2,$ and $R^3$. We will then be able to shrink $X$ and $Y$ even further. We now show that $X$ and $Y$ can be shrunk all the way down to the unit circle.

\begin{lem}
\label{lem:unit circle}
$R$ and $T$ map the unit circle to itself.
\end{lem}

\begin{proof} 
Let $(a,b)$ be a point on the unit circle. Clearly $b^2 + (-a)^2 = a^2+b^2 = 1$, so $R(a,b)$ is on the unit circle. For $T(a,b)$, we compute
\[
\left(\frac{2a+b-2}{2a+2b-3}\right)^2+\left(\frac{a+2b-2}{2a+2b-3}\right)^2
= \frac{5a^2+5b^2+8ab-12a-12b+8}{4a^2+4b^2+8ab-12a-12b+9}.
\]
Since $a^2 + b^2=1$, we can simplify this to
\[
\frac{5+8ab-12a-12b+8}{4+8ab-12a-12b+9} = 1,
\]
which shows that $T(a,b)$ is on the unit circle.
\end{proof}

It follows from Lemma \ref{lem:unit circle} and the discussion in \S \ref{sec:proj def} that the subsets
\[
X = \{(a,b) \mid a^2 + b^2 = 1, \; a,b > 0\}, \qquad Y = \{(a,b) \mid a^2 + b^2 = 1, \; a < 0 \text{ or } b < 0\}
\]
of the unit circle form a valid ping-pong table.

The projection $\rho$ is defined by
\[
\rho(x,y,z) = \left(\frac{-2(x-z)}{x-y+z}, \frac{-2y}{x-y+z}\right),
\]
so the unit circle consists of the projections of vectors $(x,y,z) \in \mathbb{R}^3$ satisfying the quadratic equation
\[
4(x-z)^2 + 4y^2 = (x-y+z)^2.
\]
Let $S$ be the surface in $\mathbb{R}^3$ defined by this equation. The maps $R$ and $T$ preserve this surface, so the intersection of $S$ with the ping-pong table in $\mathbb{R}^3$ defined in \S \ref{sec:ping-pong} is a valid ping-pong table.


\section{Comparison with the two-dimensional and four-dimensional cases}
\label{sec:2d 4d}

When $n=2$, we have
\[
R = \begin{pmatrix}
0 & -1 \\
1 & -1
\end{pmatrix}
\qquad
U = \begin{pmatrix}
0 & -1 \\
1 & 2
\end{pmatrix}
\qquad
T = \begin{pmatrix}
1 & 0 \\
-3 & 1 \\
\end{pmatrix}.
\]
As in the three-dimensional case, the matrices $U = TR$ and $RU^{-1}R^{-1} = T^{-1}R^{-1}$ have one as their only eigenvalue, and the corresponding eigenspace has dimension one. The corresponding eigenvectors are $u = (-1,1)$ and $v=(1,2)$, and one easily verifies that the open cone $C$ generated by $u$ and $v$ determines a ping-pong table by
\[
X = C \cup -C, \qquad Y = RX \cup R^2X.
\]
One can see that $C$ is (up to sign) the only simplicial cone with this property by an argument similar to the proof of Lemma \ref{lem:uv}. Note that $\ov{X} \cup \ov{Y}$ is equal to all of $\RR^2$ in this case.

When $n=4$, we have
\[
R = \begin{pmatrix}
0 & 0 & 0 & -1 \\
1 & 0 & 0 & -1 \\
0 & 1 & 0 & -1 \\
0 & 0 & 1 & -1
\end{pmatrix}
\qquad
U = \begin{pmatrix}
0 & 0 & 0 & -1 \\
1 & 0 & 0 & 4 \\
0 & 1 & 0 & -6 \\
0 & 0 & 1 & 4
\end{pmatrix}
\qquad
T = \begin{pmatrix}
1 & 0 & 0 & 0 \\
-5 & 1 & 0 & 0 \\
5 & 0 & 1 & 0 \\
-5 & 0 & 0 & 1
\end{pmatrix}.
\]
We describe a ping-pong table for $\langle R \rangle$ and $\langle T \rangle$, which is due to Brav and Thomas.

\begin{thm}[\cite{BravThomas}]
\label{thm:4D}
Let $P = \log(TR)$, and $Q = \log(T^{-1}R^{-1})$. Set $x = (0,7,-2,7)$, and define
\[
C^+ = \cone(x,Px,P^2x,P^3x), \qquad C^- = \cone(x,Qx,Q^2x,Q^3x).
\]
The sets
\[
X = \pm C^+ \cup \pm C^-, \qquad Y = RX \cup R^2X \cup R^3X \cup R^4X
\]
are a ping-pong table for $\langle R \rangle$ and $\langle T \rangle$.
\end{thm}

The proof in \cite{BravThomas} shows that
\[
T^kY \subseteq \pm C^+ \qquad \text{ and } \qquad T^{-k}Y \subseteq \pm C^-
\]
for $k > 0$. (In addition, the proof shows that $T C^+ \subseteq C^+$ and $T^{-1} C^- \subseteq C^-$.)

\begin{rem}
In \cite{BravThomas}, the matrices $R,T,U$ are represented in a different basis (and their $T$ plays the role of our $T^{-1}$). Our matrices are obtained from theirs by conjugating by the change of basis matrix
\[
S = \begin{pmatrix}
0 & 0 & 0 & 1 \\
-5 &5 & 1 & -3 \\
5 & -5 & -2 & 3 \\
0 & 5 & 1 & -1
\end{pmatrix}.
\]
The vector $x$ in Theorem \ref{thm:4D} is a positive scalar multiple of $Sv$, where
$
v = (0,1,-25/12,0)
$
is the vector defined on p.338 of their paper (in the case $d=k=5$).
\end{rem}

Explicitly, the vectors defining $C^+$ and $C^-$ are
\[
\begin{matrix}
& Px = (-5,9,-15,11) && P^3x = (-1,3,-3,1) \\
x = (0,7,-2,7) && P^2x = Q^2x = (0,1,-2,1) \\
& Qx = (5,16,-10,14) && Q^3x = (1,2,-2,4).
\end{matrix}
\]
We remark that $P^3x$ is the unique (up to scalar) eigenvector of $U = TR$, and $Q^3x$ is the unique eigenvector of $RU^{-1}R^{-1} = T^{-1}R^{-1}$. Furthermore, the matrices $P^2$ and $Q^2$ have rank 2, and their column spans intersect in the line spanned by $P^2x = Q^2x$. This vector is the analogue of $w$ in the three-dimensional case (cf. \S \ref{sec:logs}). In light of our results in the three-dimensional case, it seems natural to ask whether there is a vector $y$ such that the cone
\[
C = \cone(P^3x, Q^3x, P^2x, y)
\]
determines a ping-pong table by $X = \pm C$, $Y = RX \cup R^2X \cup R^3X \cup R^4X$. Our experiments in Sage suggest that there is no such $y$.

\bibliographystyle{alpha}
\bibliography{pingpong}

\end{document}